\documentclass[12pt]{article}
\usepackage[margin=1in]{geometry}
\usepackage{amsmath,amssymb,amsthm,mathtools}
\usepackage{bm}
\usepackage{microtype}
\usepackage[hidelinks,pdfusetitle]{hyperref}
\usepackage{enumitem}

\newcommand{\hexago}{\lozenge}

\allowdisplaybreaks

\newtheorem{theorem}{Theorem}
\newtheorem{lemma}{Lemma}
\newtheorem{proposition}{Proposition}
\newtheorem{corollary}{Corollary}
\theoremstyle{definition}

\theoremstyle{remark}
\newtheorem{remark}{Remark}

\theoremstyle{definition}
\newtheorem{problem}{Problem}

\newcommand{\R}{\mathbb{R}}
\newcommand{\Q}{\mathbb{Q}}
\newcommand{\Z}{\mathbb{Z}}

\newcommand{\cP}{\mathcal{P}}
\newcommand{\cU}{\mathcal{U}}
\newcommand{\cA}{\mathcal{A}}

\newcommand{\relint}{\operatorname{relint}}
\newcommand{\sgn}{\operatorname{sgn}}

\usepackage{hyperref}
\hypersetup{
    colorlinks=true,
    linkcolor=blue,
    urlcolor=blue
}

\title{The integer point enumerator of one irrational translate of $\cP$ is a complete invariant\\
}
\author{Sinai Robins}
\date{August 18, 2026}

\begin{document}

\maketitle

\begin{abstract}
For a polytope $\cP\subset\R^d$ and a real dilation parameter $t>0$, we study its
real-dilate counting function $L_P(t):=|tP\cap\Z^d|$.  We characterize exactly the
translation vectors $\mathbf y:= (y_1, \dots, y_d) \in\R^d$ for which the single function
$t\mapsto L_{P+\mathbf y}(t)$, taken over all $t\in \Q_{>0}$,
uniquely determines every full-dimensional rational
polytope $P\subset\R^d$.  Namely, we prove that the necessary and sufficient condition is that
$1,y_1,\ldots,y_d$ must be linearly independent over $\Q$. In particular, in each
dimension we may use the explicit algebraic vector
\[
\mathbf y^*=
\bigl(2^{1/(d+1)},2^{2/(d+1)},\ldots,2^{d/(d+1)}\bigr).
\]
With this value for $\mathbf y$, letting $t$ traverse the positive rationals suffices.
The proof is based on the discontinuities created by the facets, inspired by the work
of Royer \cite{Royer}.
\end{abstract}

\tableofcontents

\section{Introduction, and statement of the main result}

One of the most natural discrete invariants of a bounded polytope $P\subset\R^d$ is its
real-dilate {\bf lattice-point enumerator}
\begin{equation}
L_P(t):=\bigl|tP\cap\Z^d\bigr|,
\qquad t>0.
\label{eq:LP-definition}
\end{equation}
When $P$ is rational and $t$ is restricted to the positive integers, this is the
classical {\bf Ehrhart quasi-polynomial}. Allowing all positive real values of $t$ gives a
considerably more sensitive invariant. 
We consider only one fixed translation vector $\mathbf y\in\R^d$, and we study
the function
\begin{equation}
t \rightarrow L_{P+\mathbf y}(t)
:=\bigl|\Z^d\cap t(P+\mathbf y)\bigr|,
\label{eq:translated-count}
\end{equation}
for various choices of `irrational' translations $\mathbf y$, and various different
domains in $t$. If there exists a single $\mathbf y\in\R^d$ such that the function
$t \rightarrow L_{P+\mathbf y}(t)$ uniquely determines any rational polytope $P$, for
some set of dilations $t$, then we call $\mathbf y$ a {\bf universal witness} for $P$.
We will sometimes also use the words {\bf universal translate} for the same vector
$\mathbf y$.

We write $\cP_d(\Q)$ for the collection of all full-dimensional rational polytopes in
$\R^d$.
Royer~\cite{Royer} proved that if
$P,Q\in\cP_d(\Q)$ satisfy
$L_{P+\mathbf w}(t)
=
L_{Q+\mathbf w}(t)$ 
for every $\mathbf w\in\Z^d$ and every $t>0$,
then $P=Q$.
Rocha-Neves \cite{Rocha-Neves} recently found a Fourier-analytic approach that recovers
and extends the uniqueness results for rational polytopes and symmetric convex bodies,
which were established by Royer \cite{Royer}.

It is useful to distinguish the universal witnesses obtained from all positive
real dilations from those obtained using only positive rational dilations. We define:
\begin{equation*}
\cU_d^{\mathrm{Real}}
:=
\left\{
\mathbf y\in\R^d
\ \middle|\
P,Q\in\cP_d(\Q),\
L_{P+\mathbf y}(t)=L_{Q+\mathbf y}(t)\ \text{for all }t>0
\Longrightarrow P=Q
\right\}.
\end{equation*}

\begin{equation*}
\cU_d^{\mathrm{Rational}}
:=
\left\{
\mathbf y\in\R^d
\ \middle|\
P,Q\in\cP_d(\Q),\
L_{P+\mathbf y}(t)=L_{Q+\mathbf y}(t)\ \text{for all }t\in\Q_{>0}
\Longrightarrow P=Q
\right\}.
\end{equation*}

\noindent
Finally, we define:
\begin{equation}
\cA_d
:=
\bigl\{
\mathbf y=(y_1,\ldots,y_d)\in\R^d:
1,y_1,\ldots,y_d\text{ are linearly independent over }\Q
\bigr\}.
\label{eq:A-d}
\end{equation}
Equivalently, $\cA_d$ is the complement of all rational affine hyperplanes in $\R^d$.
Our first main result gives an exact characterization of a universal witness.
\begin{theorem}[Exact criterion for a universal translate]
\label{thm:main}
For every dimension $d\geq 1$, we have:
\begin{equation}
\cU_d^{\mathrm{Rational}}
=
\cU_d^{\mathrm{Real}}
=
\cA_d.
\label{eq:exact-characterization}
\end{equation}
In particular, if $1,y_1,\ldots,y_d$ are linearly independent over $\Q$, then for every
pair $P,Q\in\cP_d(\Q)$,
\begin{equation}
L_{P+\mathbf y}(t)=L_{Q+\mathbf y}(t)
\quad\text{for all }t\in\Q_{>0}
\qquad\Longrightarrow\qquad
P=Q.
\label{eq:rational-sampling-main}
\end{equation}
Conversely, if $1,y_1,\ldots,y_d$ are linearly dependent over $\Q$, then there exist
distinct polytopes $P,Q\in\cP_d(\Q)$ satisfying
\begin{equation}
L_{P+\mathbf y}(t)=L_{Q+\mathbf y}(t)
\qquad\text{for every }t>0.
\label{eq:dependent-translate-obstruction}
\end{equation}
\end{theorem}

\noindent
It is possible to choose more explicit irrational translation vectors, as follows.
\begin{corollary}[An explicit algebraic universal witness]
\label{cor:explicit-vector}
For each integer $d\geq 1$, let
\begin{equation}
\mathbf y^*
:=
\bigl(2^{1/(d+1)},2^{2/(d+1)},\ldots,2^{d/(d+1)}\bigr).
\label{eq:explicit-y}
\end{equation}
Then $\mathbf y^*\in\cU_d^{\mathrm{Rational}}$. Then the single function
\begin{equation}
t\longmapsto
\bigl|\Z^d\cap t(P+\mathbf y^*)\bigr|,
\qquad t\in\Q_{>0},
\label{eq:explicit-invariant}
\end{equation}
is a complete invariant of $P$
among the collection of full-dimensional rational polytopes in $\R^d$.
\end{corollary}

In dimension one, for example, every irrational number is a
universal witness, and translation by $\sqrt{2}$ is already sufficient.
The set of all universal witnesses also has a particularly simple topological and
measure-theoretic description.

\begin{corollary}[Size of the universal set]
\label{cor:size-universal-set}
For every $d\geq 1$, the set of universal translation vectors may be expressed as:
\begin{equation}
\cA_d
=
\bigcap_{\substack{(q_0,\ldots,q_d)\in\Z^{d+1}\\
(q_1,\ldots,q_d)\neq\mathbf 0}}
\left\{
\mathbf y\in\R^d:
q_0+q_1y_1+\cdots+q_dy_d\neq 0
\right\}.
\label{eq:A-d-Gdelta}
\end{equation}
Consequently, $\cA_d$ is a dense $G_\delta$-subset of $\R^d$, and its complement has
Lebesgue measure zero.
\end{corollary}

\begin{remark}
For the sake of intuitive understanding, we briefly outline both the necessity and
sufficiency of $\cA_d$ in the proof of Theorem~\ref{thm:main}.
First, we express a rational polytope by its primitive facet description:
\begin{equation}
P
=
\bigcap_{i=1}^{N}
\left\{
\mathbf x\in\R^d:
\langle\mathbf a_i,\mathbf x\rangle\leq b_i
\right\},
\qquad
\mathbf a_i\in\Z^d\text{ primitive},
\quad b_i\in\Q.
\label{eq:intro-facet-form}
\end{equation}
The constant $b_i$ is called a {\bf facet offset}.
It follows that 
\begin{equation}
\label{eq:intro-facet-form 2}
P + \mathbf y
=
\bigcap_{i=1}^{N}
\left\{
\mathbf x\in\R^d:
\langle\mathbf a_i,\mathbf x\rangle\leq\beta_i
\right\}, 
\end{equation}
where 
$\beta_i:=b_i+\langle\mathbf a_i,\mathbf y\rangle$.

When $\mathbf y\in\cA_d$, all of the numbers $\beta_i$ are nonzero, and the ratio of
two distinct facet offsets is never rational. Thus two different facets cannot meet the
lattice at the same dilation parameter. On the other hand, each individual facet meets
the lattice at infinitely many sufficiently large dilation parameters. Its lattice points
therefore create an infinite sequence of isolated jumps in~\eqref{eq:translated-count}.
The location and direction of such a jump identify the corresponding primitive rational
facet inequality. This gives the sufficiency of~\eqref{eq:A-d}.

The converse has a different character. A rational affine relation among
$1,y_1,\ldots,y_d$ allows us to construct a nonidentity unimodular linear map $U$ for
which $(U-I)\mathbf y$ is rational. If
\begin{equation}
Q:=UP+(U-I)\mathbf y,
\label{eq:intro-affine-obstruction}
\end{equation}
then
\begin{equation}
Q+\mathbf y=U(P+\mathbf y).
\label{eq:intro-lattice-equivalence}
\end{equation}
The unimodular map $U$ makes the two translated counting functions identical.
\hfill $\hexago$
\end{remark}

The recent work of Higashitani, Murai, and Yoshinaga
\cite{Higashitani-Murai-Yoshinaga} also contains a reconstruction theorem closely
related to Royer's. The authors of \cite{Higashitani-Murai-Yoshinaga} show that the
family of Ehrhart quasi-polynomials
$
\left\{
L_{\mathcal P+\mathbf v}(t)
\ | \
\mathbf v\in\mathbb Q^d
\right\}
$
with the dilation parameter $t$ restricted to nonnegative integers, determines a
rational polytope \(\mathcal P\) up to integral translations. As they note, a stronger
form had previously appeared in Alhajjar’s thesis \cite{Alhajjar2017}.

The paper is structured as follows. In Section~\ref{sec:facets}, we analyze the facet
offsets and the resulting one-sided jumps. In Section~\ref{sec:exact-proof}, we prove the
exact characterization, the affine-lattice obstruction, and the explicit algebraic
corollary.

\section{Lemmas for the facet offsets, and isolated discontinuities}
\label{sec:facets}

We begin with the elementary arithmetic fact that makes the facet argument possible.
The conclusion applies equally well to facets belonging to two different rational
polytopes.

\begin{lemma}[Rational ratios and incommensurability of translated facet offsets]
\label{lem:rational-ratio}
Let $\mathbf y\in\cA_d$. Let $\mathbf a,\mathbf c\in\Z^d$ be nonzero primitive
vectors, and let $b,e\in\Q$. Define
\begin{equation}
\beta:=b+\langle\mathbf a,\mathbf y\rangle,
\qquad
\delta:=e+\langle\mathbf c,\mathbf y\rangle.
\label{eq:beta-delta}
\end{equation}
Then the following statements hold.
\begin{enumerate}[label=\textup{(\alph*)}]
\item
We have $\beta\neq 0$ and $\delta\neq 0$.

\item
If
$\frac{\beta}{\delta}\in\Q$,
then
\begin{equation}
(\mathbf a,b)=(\mathbf c,e)
\qquad\text{or}\qquad
(\mathbf a,b)=-(\mathbf c,e).
\label{eq:plus-minus-facets}
\end{equation}

\item
Under the hypothesis of part~\textup{(b)}, if $\beta$ and $\delta$ have the same
sign, then only the first alternative in~\eqref{eq:plus-minus-facets} is possible.

\item
Suppose that $P\in\cP_d(\Q)$ has the irredundant primitive facet description
\eqref{eq:intro-facet-form}. We consider the $i$'th facet defined by 
$F_i
:=
P\cap
\left\{
\mathbf x\in\R^d:
\langle\mathbf a_i,\mathbf x\rangle=b_i
\right\}$.
We note that 
$
P + \mathbf y
=
\bigcap_{i=1}^{N}
\left\{
\mathbf x\in\R^d:
\langle\mathbf a_i,\mathbf x\rangle\leq\beta_i
\right\}$,
where $\beta_i:=b_i+\langle\mathbf a_i,\mathbf y\rangle$.

Then each $\beta_i$ is nonzero, and the translated facet offsets are pairwise
incommensurable:
\begin{equation}
\frac{\beta_i}{\beta_j}\notin\Q,
\label{eq:pairwise-irrational-ratios}
\end{equation}
for all $i \not= j$.
\end{enumerate}
\end{lemma}

\begin{proof}
\noindent
\textup{(a)}
If $\beta=0$, then
$b+a_1y_1+\cdots+a_dy_d=0$
is a nontrivial rational linear relation among $1,y_1,\ldots,y_d$, contrary to
$\mathbf y\in\cA_d$. Thus $\beta\neq 0$, and the same argument gives
$\delta\neq 0$.

\medskip
\noindent
\textup{(b)}
Suppose that $\beta=r\delta$ for some $r\in\Q$. Expanding this equality gives
\begin{equation}
(b-re)+\langle\mathbf a-r\mathbf c,\mathbf y\rangle=0.
\label{eq:ratio-affine-relation}
\end{equation}
The defining property of $\cA_d$ implies that
$\mathbf a=r\mathbf c$ and $b=re$. Since $\mathbf a$ and $\mathbf c$ are primitive
integer vectors, their rational proportionality forces
\begin{equation}
r=1
\qquad\text{or}\qquad
r=-1.
\label{eq:r-plus-minus-one}
\end{equation}
This proves~\eqref{eq:plus-minus-facets}.

\medskip
\noindent
\textup{(c)}
If $\beta$ and $\delta$ have the same sign, then
$r=\beta/\delta>0$. Hence $r=1$, and only the first alternative in
\eqref{eq:plus-minus-facets} is possible.

\medskip
\noindent
\textup{(d)}
 Part~\textup{(a)} gives $\beta_i\neq 0$ for
every $i$. Now fix $i\neq j$ and suppose that $\beta_i/\beta_j\in\Q$.
Part~\textup{(b)}, applied to $(\mathbf a_i,b_i)$ and $(\mathbf a_j,b_j)$, gives
\begin{equation}
(\mathbf a_i,b_i)=(\mathbf a_j,b_j)
\qquad\text{or}\qquad
(\mathbf a_i,b_i)=-(\mathbf a_j,b_j).
\label{eq:lemma-plus-minus-facet-data}
\end{equation}
The first alternative repeats the same facet inequality, contradicting irredundancy.
The second forces $P$ to satisfy both
$\langle\mathbf a_i,\mathbf x\rangle\leq b_i$ and
$\langle\mathbf a_i,\mathbf x\rangle\geq b_i$, placing $P$ in the hyperplane
$\langle\mathbf a_i,\mathbf x\rangle=b_i$ and contradicting the full-dimensionality
of $P$. Therefore~\eqref{eq:pairwise-irrational-ratios} holds.
\end{proof}

\bigskip
We next describe the discontinuities of the counting function for $K:= P + \mathbf y$. For $t_0>0$, we define
\begin{equation}
L_K(t_0^-)
:=
\lim_{t\uparrow t_0}L_K(t),
\qquad
L_K(t_0^+)
:=
\lim_{t\downarrow t_0}L_K(t),
\label{eq:one-sided-limits}
\end{equation}
and we also define the signed jump
\begin{equation}
Jump_{K}(t_0)
:=
L_K(t_0^+)-L_K(t_0^-).
\label{eq:signed-jump}
\end{equation}

\begin{lemma}[The one-facet jump formula]
\label{lem:one-facet-jump}
Let $P\in\cP_d(\Q)$ and $\mathbf y\in\cA_d$, and let $K:=P+\mathbf y$ be written as
in~\eqref{eq:intro-facet-form}. Then the set of dilations that measures boundary integer points:
\begin{equation}
\mathcal D_K
:=
\bigl\{t>0:\partial(tK)\cap\Z^d\neq\varnothing\bigr\}
\label{eq:D-K}
\end{equation}
is locally finite. If $t_0\in\mathcal D_K$, then there is a unique index $i$ such that
\begin{equation}
\partial(t_0K)\cap\Z^d
\subseteq
\relint(t_0(F_i+\mathbf y)).
\label{eq:boundary-one-facet}
\end{equation}
Moreover,
\begin{equation}
Jump_{K}(t_0)
=
\sgn(\beta_i)
\bigl|\partial(t_0K)\cap\Z^d\bigr|
\neq 0.
\label{eq:jump-formula}
\end{equation}
Thus every lattice contact with the boundary gives a genuine jump, and the sign of the
jump is the sign of the corresponding translated facet offset.
\end{lemma}

\begin{proof}
Suppose that a lattice point $\mathbf m$ lies on the $i$-th facet of $tK$. Then
\begin{equation}
\langle\mathbf a_i,\mathbf m\rangle=t\beta_i.
\label{eq:lattice-contact-equation}
\end{equation}
The left-hand side is an integer, so every possible contact time for the $i$-th facet has
the form
\begin{equation}
t=\frac{n}{\beta_i},
\qquad
n\in\Z,
\qquad
\frac{n}{\beta_i}>0.
\label{eq:contact-times}
\end{equation}
Only finitely many such values lie in a fixed compact interval of $(0,\infty)$.
Since there are only finitely many facets, $\mathcal D_K$ is locally finite. On an
interval containing no point of $\mathcal D_K$, no lattice point can change its membership
in $tK$, so $L_K(t)$ is constant there. In particular, the one-sided limits
in~\eqref{eq:one-sided-limits} exist.

If a time $t_0$ were associated with two distinct facets $i$ and $j$, then for some
nonzero integers $n_i,n_j$ we would have
\begin{equation}
t_0\beta_i=n_i,
\qquad
t_0\beta_j=n_j,
\label{eq:two-facet-contact}
\end{equation}
and hence
\begin{equation}
\frac{\beta_i}{\beta_j}
=
\frac{n_i}{n_j}
\in\Q,
\label{eq:two-facet-ratio}
\end{equation}
contrary to~\eqref{eq:pairwise-irrational-ratios}. Thus only one facet can meet the
lattice at $t_0$. A lattice point cannot lie on a lower-dimensional face, because such a
point would lie on at least two facets. This proves~\eqref{eq:boundary-one-facet}.

It remains to compute the direction of the jump. Let
$\mathbf m\in\partial(t_0K)\cap\Z^d$. All inequalities other than the $i$-th are strict
at $\mathbf m$, and hence remain strict for $t$ sufficiently close to $t_0$. The active
inequality is
\begin{equation}
\langle\mathbf a_i,\mathbf m\rangle\leq t\beta_i.
\label{eq:active-inequality}
\end{equation}
If $\beta_i>0$, then $\mathbf m$ is outside $tK$ immediately before $t_0$ and inside
$tK$ immediately after $t_0$. If $\beta_i<0$, the reverse occurs. Every boundary
lattice point lies on the same facet and therefore moves in the same direction. Summing
their contributions gives~\eqref{eq:jump-formula}.
\end{proof}

The next observation explains why knowing the values only at rational dilation parameters
is enough. It is simply the density of $\Q$, together with the local finiteness in
Lemma~\ref{lem:one-facet-jump}.

\begin{lemma}[Rational samples detect every jump]
\label{lem:rational-samples-detect-jumps}
Let $P,Q\in\cP_d(\Q)$ and $\mathbf y\in\cA_d$. Suppose that
\begin{equation}
L_{P+\mathbf y}(t)=L_{Q+\mathbf y}(t)
\qquad\text{for every }t\in\Q_{>0}.
\label{eq:rational-sample-equality}
\end{equation}
If $t_0$ is a jump time for $L_{P+\mathbf y}$, then it is also a jump time for
$L_{Q+\mathbf y}$, and the two signed jumps are equal.
\end{lemma}

\begin{proof}
By local finiteness, the functions $L_{P+\mathbf y}$ and
$L_{Q+\mathbf y}$ have one-sided limits at $t_0$. Choose rational sequences
\begin{equation}
r_n\uparrow t_0,
\qquad
s_n\downarrow t_0.
\label{eq:rational-sequences}
\end{equation}
The assumption~\eqref{eq:rational-sample-equality} gives
$L_{P+\mathbf y}(r_n)=L_{Q+\mathbf y}(r_n)$ and
$L_{P+\mathbf y}(s_n)=L_{Q+\mathbf y}(s_n)$
for all $n$. Passing to the one-sided limits yields
\begin{equation}
L_{P+\mathbf y}(t_0^-)=L_{Q+\mathbf y}(t_0^-),
\qquad
L_{P+\mathbf y}(t_0^+)=L_{Q+\mathbf y}(t_0^+).
\label{eq:one-sided-equality}
\end{equation}
Since the two limits for $P+\mathbf y$ are different, the two limits for
$Q+\mathbf y$ are different as well. Thus $t_0$ is a jump time for $Q+\mathbf y$,
and~\eqref{eq:one-sided-equality} gives equality of the signed jumps.
\end{proof}

The final geometric ingredient will show that every facet is visible in the counting function.
One of the main ingredients in the proof of the following Lemma is the covering radius of the sublattice in an integral affine
hyperplane.


\begin{lemma}[Every facet produces infinitely many jumps]
\label{lem:facet-visible}
Let $P\in\cP_d(\Q)$ and $\mathbf y\in\cA_d$. Let $F$ be any facet of
$P$, written as
$F
=
P\cap
\left\{
\mathbf x\in\R^d:
\langle\mathbf a,\mathbf x\rangle=b
\right\}$,
where $\mathbf a\in\Z^d$ is primitive, $b\in\Q$, and
\begin{equation}
\beta
:=
b+\langle\mathbf a,\mathbf y\rangle.
\label{eq:fixed-translated-offset}
\end{equation}
Then $\beta\neq 0$, and there exists an integer
$N=N(F,\mathbf y)\geq 1$ such that
\begin{equation}
\relint\left(
\frac{n}{\beta}(F+\mathbf y)
\right)
\cap\Z^d
\neq\varnothing
\label{eq:facet-lattice-point}
\end{equation}
for every integer $n$ satisfying
\begin{equation}
n\beta>0,
\qquad
|n|\geq N.
\label{eq:large-eligible-n}
\end{equation}
Consequently, the translated facet $F+\mathbf y$ produces a nonzero
jump at every sufficiently large eligible time
\begin{equation}
t
=
\frac{n}{\beta}
>
0.
\label{eq:eligible-time}
\end{equation}
In particular, every facet produces infinitely many nonzero jumps.
\end{lemma}

\begin{proof}
Fix the facet $F$. All the notation introduced below depends on this
fixed facet and on $\mathbf y$; we suppress this dependence.
For any $s\in\R$, define the affine hyperplane
\begin{equation}
H_s
:=
\left\{
\mathbf x\in\R^d:
\langle\mathbf a,\mathbf x\rangle=s
\right\}.
\label{eq:affine-hyperplane-H-s}
\end{equation}
Thus the affine hull of $F+\mathbf y$ is $H_\beta$. Choose
$\mathbf c\in\relint(F+\mathbf y)$ and $\rho>0$ such that the ball
\begin{equation}
B_{H_\beta}(\mathbf c,\rho)
\subseteq
\relint(F+\mathbf y).
\label{eq:relative-inball}
\end{equation}

Geometrically, the argument compares two scales. Under dilation, the
relative ball in~\eqref{eq:relative-inball} grows linearly, while the
covering radius of the lattice in its supporting hyperplane remains
fixed. Once the ball is sufficiently large, it must contain a lattice
point.

Indeed, let $n\in\Z$ satisfy $n\beta>0$ and put
$t
:=
\frac{n}{\beta}$.
Then $t>0$ and
\begin{equation}
tH_\beta
=
H_n.
\label{eq:scaled-hyperplane}
\end{equation}
Scaling~\eqref{eq:relative-inball} by $t$ gives
\begin{equation}
B_{H_n}
\left(
t\mathbf c,
t\rho
\right)
\subseteq
\relint\left(
t(F+\mathbf y)
\right).
\label{eq:scaled-relative-inball}
\end{equation}
Because $\mathbf a$ is primitive, there exists
$\mathbf z\in\Z^d$ such that $\langle\mathbf a,\mathbf z\rangle=1$.
Therefore $H_n\cap\Z^d
=
n\mathbf z+\Lambda$, where
$\Lambda
:=
\left\{
\mathbf m\in\Z^d:
\langle\mathbf a,\mathbf m\rangle=0
\right\}$.
The lattice $\Lambda$ has rank $d-1$ in $H_0$. Let $R$ be its
covering radius. Since \(H_n\cap\Z^d\) is a translate of $\Lambda$,
the same covering radius $R$ works in every $H_n$.

Choose an integer $N\geq 1$ sufficiently large that
\begin{equation}
\frac{N\rho}{|\beta|}
>
R.
\label{eq:choice-of-N}
\end{equation}
Whenever $|n|\geq N$, we have
\begin{equation}
t\rho
=
\frac{|n|\rho}{|\beta|}
\geq
\frac{N\rho}{|\beta|}
>
R.
\label{eq:radius-exceeds-covering-radius}
\end{equation}
Hence the relative ball in~\eqref{eq:scaled-relative-inball} contains
a point of $H_n\cap\Z^d$. Consequently,
\begin{equation}
\relint\left(
t(F+\mathbf y)
\right)
\cap\Z^d
\neq\varnothing.
\label{eq:relative-interior-lattice-contact}
\end{equation}
Since $t=n/\beta$, this proves~\eqref{eq:facet-lattice-point}.
Lemma~\ref{lem:one-facet-jump} now shows that this lattice contact
produces a nonzero jump at $t$.

Finally, there are infinitely many integers $n$ having the same sign
as $\beta$ and satisfying $|n|\geq N$. Thus $F+\mathbf y$ produces
infinitely many nonzero jumps.
\end{proof}

\section{Proof of Theorem \ref{thm:main}, and an explicit vector}
\label{sec:exact-proof}

We first prove the obstruction needed for the converse direction of
Theorem~\ref{thm:main}.

\begin{proposition}[The affine-lattice obstruction]
\label{prop:affine-lattice-obstruction}
Let $d\geq 1$ and let $\mathbf y\in\R^d$. Suppose that
$1,y_1,\ldots,y_d$ are linearly dependent over $\Q$. Then there exist distinct
polytopes $P,Q\in\cP_d(\Q)$ such that
\begin{equation}
L_{P+\mathbf y}(t)=L_{Q+\mathbf y}(t)
\qquad\text{for every }t>0.
\label{eq:obstruction-count-equality}
\end{equation}
In particular, $\mathbf y\notin\cU_d^{\mathrm{Real}}$. Consequently, for every
$d\geq 1$,
\begin{equation}
\cU_d^{\mathrm{Real}}
\subseteq
\cA_d.
\label{eq:U-real-subset-A}
\end{equation}
\end{proposition}

\begin{proof}
After clearing denominators in a nontrivial rational relation, we may assume that
$\langle\mathbf a,\mathbf y\rangle=b$, with
$\mathbf a\in\Z^d\setminus\{\mathbf 0\}$ and
$b\in\Z$.

\medskip
\noindent
{\bf Case $d=1$}. \ For dimension $d=1$, it follows that
$y\in\Q$.
We choose any rational interval $P$ satisfying
$P\neq -P-2y$,
and we define
\begin{equation}
Q:=-P-2y.
\label{eq:one-dimensional-Q}
\end{equation}
Then $Q$ is rational, $Q\neq P$, and
\begin{equation}
Q+y=-(P+y).
\label{eq:one-dimensional-reflection}
\end{equation}
The map $m\mapsto -m$ is a bijection of $\Z$, so~\eqref{eq:obstruction-count-equality}
follows.

\medskip
\noindent
{\bf Case $d\geq 2$}. \ Choose a nonzero vector $\mathbf v\in\Z^d$ satisfying
\begin{equation}
\langle\mathbf a,\mathbf v\rangle=0.
\label{eq:v-orthogonal-a}
\end{equation}
Such a vector exists because the kernel of the homomorphism
$\Z^d\to\Z$, $\mathbf v\mapsto\langle\mathbf a,\mathbf v\rangle$, has rank at least
$d-1$. Define
\begin{equation}
U:=I+\mathbf v\mathbf a^{T}.
\label{eq:unipotent-U}
\end{equation}
Since $\mathbf a^T\mathbf v=0$, the matrix determinant lemma and a direct multiplication
give
\begin{equation}
\det U=1+\mathbf a^T\mathbf v=1,
\qquad
U^{-1}=I-\mathbf v\mathbf a^T.
\label{eq:U-unimodular}
\end{equation}
Thus $U\in\operatorname{SL}_d(\Z)$, and $U\neq I$. Moreover,
\begin{equation}
(U-I)\mathbf y
=
\mathbf v\langle\mathbf a,\mathbf y\rangle
=
b\mathbf v
\in\Z^d.
\label{eq:U-minus-I-y}
\end{equation}

We now consider the nonidentity rational affine map
\begin{equation}
T(\mathbf x):=U\mathbf x+b\mathbf v.
\label{eq:affine-map-T}
\end{equation}
Choose a rational point $\mathbf p$ not fixed by $T$, and then choose a sufficiently small
full-dimensional rational cube $P$ centered at $\mathbf p$. Its image
\begin{equation}
Q:=T(P)=UP+b\mathbf v
\label{eq:Q-affine-image}
\end{equation}
is a full-dimensional rational polytope. The barycenters of $P$ and $Q$ are
$\mathbf p$ and $T(\mathbf p)$, respectively, so $P\neq Q$. By
\eqref{eq:U-minus-I-y},
\begin{equation}
Q+\mathbf y
=
UP+b\mathbf v+\mathbf y
=
UP+U\mathbf y
=
U(P+\mathbf y).
\label{eq:Q-y-U-P-y}
\end{equation}
Therefore
\begin{align}
L_{Q+\mathbf y}(t)
&=
\bigl|\Z^d\cap tU(P+\mathbf y)\bigr|
\label{eq:unimodular-count-1}
\\
&=
\bigl|\Z^d\cap U\bigl(t(P+\mathbf y)\bigr)\bigr|
\label{eq:unimodular-count-2}
\\
&=
\bigl|\Z^d\cap t(P+\mathbf y)\bigr|
=
L_{P+\mathbf y}(t),
\label{eq:unimodular-count-3}
\end{align}
where the last equality uses the fact that $U$ is a bijection of $\Z^d$.
\end{proof}

\begin{proof}[Proof of Theorem~\ref{thm:main}]
We first prove the containment:
\begin{equation}
\cA_d\subseteq\cU_d^{\mathrm{Rational}}.
\label{eq:A-subset-U-rat}
\end{equation}
Fix $\mathbf y\in\cA_d$, and let $P,Q\in\cP_d(\Q)$ satisfy
\eqref{eq:rational-sampling-main}. We express the given irredundant primitive facet
descriptions:
\begin{equation}
\label{eq:P-facet-presentation}
P
=
\bigcap_{i=1}^{N}
\left\{
\mathbf x:\langle\mathbf a_i,\mathbf x\rangle\leq b_i
\right\},
\quad
Q
=
\bigcap_{j=1}^{M}
\left\{
\mathbf x:\langle\mathbf c_j,\mathbf x\rangle\leq e_j
\right\},
\end{equation}
where $\mathbf a_i,\mathbf c_j\in\Z^d$ are primitive and
$b_i,e_j\in\Q$. Set
\begin{equation}
\beta_i:=b_i+\langle\mathbf a_i,\mathbf y\rangle,
\qquad
\delta_j:=e_j+\langle\mathbf c_j,\mathbf y\rangle.
\label{eq:beta-i-delta-j}
\end{equation}

Fix one facet of $P$, indexed by $i$. By Lemma~\ref{lem:facet-visible}, choose an integer
$n$ of the same sign as $\beta_i$, and sufficiently large in absolute value, such that
\begin{equation}
t_0:=\frac{n}{\beta_i}>0
\label{eq:chosen-jump-time}
\end{equation}
is a jump time for $L_{P+\mathbf y}$. Lemma~\ref{lem:rational-samples-detect-jumps}
shows that $t_0$ is also a jump time for $L_{Q+\mathbf y}$, with the same signed jump.
By Lemma~\ref{lem:one-facet-jump}, the lattice points on the boundary of
$t_0(Q+\mathbf y)$ lie in the relative interior of one unique facet, say the $j$-th
facet. Hence, for some nonzero integer $m$,
\begin{equation}
t_0\delta_j=m.
\label{eq:Q-contact-integer}
\end{equation}
Together with $t_0\beta_i=n$, this gives
\begin{equation}
\frac{\beta_i}{\delta_j}
=
\frac{n}{m}
\in\Q.
\label{eq:matching-rational-ratio}
\end{equation}
The equality of the signed jumps and~\eqref{eq:jump-formula} imply that
$\beta_i$ and $\delta_j$ have the same sign. Lemma~\ref{lem:rational-ratio} therefore
gives
\begin{equation}
\mathbf a_i=\mathbf c_j,
\qquad
b_i=e_j.
\label{eq:matching-facet-data}
\end{equation}
Thus every facet inequality of $P$ occurs among the facet inequalities of $Q$.
Interchanging $P$ and $Q$ gives the reverse inclusion. Since the two presentations are
irredundant, we have $P=Q$.
This proves~\eqref{eq:A-subset-U-rat}.

To finish the proof, we first note that we trivially have
$\cU_d^{\mathrm{Rational}}
\subseteq
\cU_d^{\mathrm{Real}}$.
Combining this with
\eqref{eq:A-subset-U-rat}, 
we obtain
\begin{equation}
\cA_d\subseteq\cU_d^{\mathrm{Rational}}
\subseteq
\cU_d^{\mathrm{Real}}
\subseteq
\cA_d,
\label{eq:universal-translate-containment-chain}
\end{equation}
where the last containment follows from
Proposition~\ref{prop:affine-lattice-obstruction}.
\end{proof}

\section
{Proofs of Corollary \ref{cor:explicit-vector} and
Corollary~\ref{cor:size-universal-set}
}
\begin{proof}[Proof of Corollary~\ref{cor:explicit-vector}]
Let $f_d(X):=X^{d+1}-2\in\Z[X]$.
We apply Eisenstein's criterion at the prime $2$. Its three hypotheses are especially
transparent here.  Namely, the leading coefficient of $f_d$ is $1$, 
every non-leading coefficient is divisible by $2$, and the constant coefficient $-2$ is not divisible by $2^2$.
Therefore Eisenstein's criterion implies that
$f_d(X)=X^{d+1}-2$ is irreducible over $\Q$.
The positive real number $\alpha_d=2^{1/(d+1)}$ is a root of $f_d$, so its minimal
polynomial over $\Q$ is $f_d$. In particular,
\begin{equation}
[\Q(\alpha_d):\Q]=d+1.
\label{eq:degree-alpha}
\end{equation}
If $1,\alpha_d,\ldots,\alpha_d^d$ were linearly dependent over $\Q$, there would be a
nonzero polynomial
\begin{equation}
g(X)=q_0+q_1X+\cdots+q_dX^d\in\Q[X]
\label{eq:lower-degree-polynomial}
\end{equation}
of degree at most $d$ satisfying $g(\alpha_d)=0$. The minimal polynomial $f_d$ would
then divide $g$, which is impossible because
\begin{equation}
\deg f_d=d+1>\deg g.
\label{eq:degree-contradiction}
\end{equation}
Hence
$1,\alpha_d,\alpha_d^2,\ldots,\alpha_d^d$
are linearly independent over $\Q$, which is equivalent to the statement
$\mathbf y^*\in\cA_d$. The conclusion now follows from Theorem~\ref{thm:main}.
\end{proof}

\bigskip
\begin{proof}[Proof of Corollary~\ref{cor:size-universal-set}]
Equation~\eqref{eq:A-d-Gdelta} is simply the definition of rational linear independence
after clearing denominators. For every integer vector $(q_0,\ldots,q_d)$ with
$(q_1,\ldots,q_d)\neq\mathbf 0$, the set
\begin{equation}
\left\{
\mathbf y\in\R^d:
q_0+q_1y_1+\cdots+q_dy_d=0
\right\}
\label{eq:rational-affine-hyperplane}
\end{equation}
is a proper closed affine hyperplane. Its complement is open and dense, and the
intersection in~\eqref{eq:A-d-Gdelta} is countable. The Baire category theorem tells us
that $\cA_d$ is a dense $G_\delta$-set. Each hyperplane in
\eqref{eq:rational-affine-hyperplane} has Lebesgue measure zero, so their countable union
also has measure zero.
\end{proof}

\begin{remark}
The proof gives more than an existence theorem. For $\mathbf y\in\cA_d$, each facet
of $P$ produces an infinite sequence of isolated jump times
\begin{equation}
\left\{
\frac{n}{b_i+\langle\mathbf a_i,\mathbf y\rangle}:
 n\in\Z,
\ n\bigl(b_i+\langle\mathbf a_i,\mathbf y\rangle\bigr)>0,
\ |n|\gg 1
\right\}.
\label{eq:facet-jump-sequence}
\end{equation}
Different facets produce incommensurable sequences. The proof of
Theorem~\ref{thm:main} uses this separation only qualitatively, but it suggests a more
effective reconstruction procedure based on clustering the discontinuities into their
facet sequences.
\hfill $\hexago$
\end{remark}

\begin{remark}[Relation with Royer's reconstruction method]
\label{rem:Royer-method}
For $\mathbf y\in\cA_d$, the possible contact times of the $i$-th facet
are contained in the additive subgroup
\begin{equation}
\Gamma_i
:=
\frac{1}{\beta_i}\Z,
\qquad
\beta_i=b_i+\langle\mathbf a_i,\mathbf y\rangle.
\label{eq:facet-contact-group}
\end{equation}
If $i\neq j$, then
$\Gamma_i\cap\Gamma_j=\{0\}$.
Indeed, a nonzero element of the intersection would give integers
$m,n\neq 0$ such that
$\frac{n}{\beta_i}
=
\frac{m}{\beta_j}$,
and hence
\begin{equation}
\frac{\beta_i}{\beta_j}
=
\frac{n}{m}
\in\Q,
\label{eq:common-time-rational-ratio}
\end{equation}
contrary to~\eqref{eq:pairwise-irrational-ratios}. Thus the positive
contact-time sequences of two distinct facets are tails of
incommensurable arithmetic progressions; equivalently, their spacings
$1/|\beta_i|$ and $1/|\beta_j|$ have irrational ratio.

The discontinuity viewpoint used here is closely related to the
reconstruction method of Royer~\cite{Royer}. Royer's Lemma~8 identifies
left discontinuities of the real-parameter Ehrhart function with lattice
points entering through the front facets, and right discontinuities with
lattice points leaving through the back facets. Under the separation
condition~\eqref{eq:pairwise-irrational-ratios}, this becomes the
one-facet jump formula of Lemma~\ref{lem:one-facet-jump}.

Royer's Lemma~9 proves the stronger asymptotic statement that, along the
dilation parameters for which the affine span of a semi-rational facet
meets the lattice, its lattice-point count, divided by the appropriate
power of the dilation parameter, converges to its relative volume,
uniformly in the integer translation. The present proof requires only
the weaker consequence that every sufficiently large eligible dilate of
a genuine facet contains a lattice point. In
Lemma~\ref{lem:facet-visible}, we prove this consequence directly by
using the covering radius of the lattice in the supporting hyperplane.
\hfill $\hexago$
\end{remark}

\section{Further remarks, and open problems}

The proof of Theorem~\ref{thm:main} is qualitative but contains the beginnings of an
explicit reconstruction method. Each facet gives a sequence of jump times with an
asymptotic spacing determined by the translated offset
$b_i+\langle\mathbf a_i,\mathbf y\rangle$, and Lemma~\ref{lem:rational-ratio} shows that
no two distinct rational facets can generate commensurable sequences.

\begin{problem}[Effective facet reconstruction]
Consider the explicit vector $\mathbf y^*$ in~\eqref{eq:explicit-y}. Given suitable
bounds on the denominators, coordinates, and number of facets of an unknown rational
polytope $P$, recover the primitive facet inequalities of $P$ from finitely many values of
\begin{equation}
L_{P+\mathbf y^*}(t),
\qquad
t\in\Q_{>0},
\label{eq:finite-reconstruction-data}
\end{equation}
and obtain an explicit upper bound for the required dilation parameters.
\end{problem}

The covering-radius proof of Lemma~\ref{lem:facet-visible} gives a natural geometric
starting point. A quantitative theorem would require effective control of the relative
inradii of the facets and the covering radii of their integral hyperplane lattices, followed
by a procedure that groups the observed jumps into the incommensurable families
in~\eqref{eq:facet-jump-sequence}.

\begin{problem}[Stability]
Develop a stable version of Theorem~\ref{thm:main}. Suppose that, on a
long finite interval, the discontinuities of two translated counting
functions can be matched so that the corresponding jump times and jump
sizes differ by prescribed small amounts. Under suitable a priori
bounds on the rational polytopes, determine whether this forces
quantitative closeness of their facet normals and facet offsets.
\end{problem}

\bigskip
\noindent
{\bf AI Usage}

\medskip
\noindent
The author used ChatGPT 5.6 Sol for editorial assistance, proofreading, suggestions, and some discovery, during the writing of this paper.  The author independently provided and re-checked all
mathematical arguments, statements, and references, and takes full responsibility for the final manuscript.


\end{document}